\documentclass{amsart}
\usepackage{graphicx}

\usepackage{cmap}
\usepackage[utf8]{inputenc}
\usepackage[T1]{fontenc}
\usepackage[brazilian,english]{babel}
\usepackage{csquotes}
\usepackage{indentfirst}

\newtheorem{theorem}{Theorem}[section]
\newtheorem{lemma}[theorem]{Lemma}

\newtheorem{cor}[theorem]{Corollary}

\numberwithin{equation}{section}

\begin{document}

\title{Monotonicity of Zeros of Jacobi-Angelesco polynomials}

\author{Eliel J. C. dos Santos}
\address{IMECC, Universidade Estadual de Campinas,
Campinas-SP, Brasil}
\email{elielubarana@gmail.com}
\thanks{Research supported by the Brazilian Science Foundation CAPES}


%\date{\ }



\begin{abstract}
We study the monotonic behaviour of the zeros of the multiple Jacobi-Angelesco orthogonal polynomials, in the diagonal case,  with respect to the parameters $\alpha,\beta$ and $\gamma$.
We prove that the zeros are monotonic functions of $\alpha$ and $\gamma$ and consider some special cases of how the zeros depend on $\beta$, especially in the presence of symmetry.
As a consequence we obtain results about monotonicity of zeros of Jacobi-Laguerre and Laguerre-Hermite multiple orthogonal polynomials too.
\end{abstract}

\maketitle

\section{Introduction}
There are two types of multiple orthogonal polynomials:
\begin{itemize}
\item[Type I:]
For a multi-index $\overrightarrow{n}=(n_1,\ldots,n_r)$ the vector $\left(A_{\overrightarrow{n},1}(x),\ldots,A_{\overrightarrow{n},r}(x)\right)$ of $r$ polynomials, where $A_{\overrightarrow{n},j}(x)$ has degree at most $n_j-1$,
is called a type I system of multiple orthogonal polynomials if
\begin{eqnarray}\label{def_tipoI}
\sum_{j=1}^{r}\int x^k A_{\overrightarrow{n},j}(x)d\mu_j(x)=0,\quad k=0,1,\ldots,|\overrightarrow{n}|-2,
\end{eqnarray}
and
$$\sum_{j=1}^{r}\int x^{|\overrightarrow{n}|-1} A_{\overrightarrow{n},j}(x)d\mu_j(x)=1.$$
\item[Type II:] A monic polynomial $P_{\overrightarrow{n}}(x)$ of degree $|\overrightarrow{n}|$ is said to be a type II multiple orthogonal polynomial if
\begin{eqnarray}\label{def_TipoII}
\int x^k P_{\overrightarrow{n}}(x)d\mu_j(x)=0\ \  \mathrm{for}\ \ k=0,1,\ldots,n_j-1\ \  \mathrm{and}\ \ j=1,\ldots,r.
\end{eqnarray}
\end{itemize}
An Angelesco system is a set of $r$ different measures $\left(\mu_1,\ldots,\mu_r\right)$ such that the convex hull of the support of each measure $d\mu_i$ is an interval $[a_i,b_i]$ and the open intervals $(a_i,b_i)$ are disjoint.
The following fundamental result holds (see  \cite[pag. 609]{Mourad}):
\begin{theorem}\label{loc_zeros_angelesco}
Let $P_{\overrightarrow{n}}$ be a multiple orthogonal polynomial polynomial of type II with   $\overrightarrow{n}=(n_1,\ldots,n_r)$ and a the corresponding measures $(\mu_1,\ldots,\mu_r)$.
Suppose that the support of each measure $\mu_i$ has infinite points. Then $P_{\overrightarrow{n}}$ has $n_i$ zeros in each interval $(a_i,b_i)$, for $i=1,\ldots,r$.
\end{theorem}
According with this theorem, $P_{\overrightarrow{n}}$ can be represented in the form
$$
P_{\overrightarrow{n}}(x)=\prod_{i=1}^{r} q_{\overrightarrow{n},i}(x),
$$
where $q_{\overrightarrow{n},i}$ are polynomials with degree $n_i$ and their zeros belong to the respective  intervals of orthogonality $(a_i,b_i).$

In this paper we are interested in the behaviour of the zeros of the multiple Jacobi-Angelesco, Jacobi-Laguerre and Laguerre-Hermite orthogonal polynomials and first we recall
their definitions (see \cite{Walter}).
The Jacobi-Angelesco polynomials, denoted by $P_{n,m}^{(\alpha,\beta,\gamma)}(x;a),$ are multiple orthogonal polynomials, with respect to the weight functions  $\omega_1(x)=(x-a)^{\alpha}|x|^{\beta}(1-x)^{\gamma}$ in the interval $[a,0]$ and $\omega_2(x)=(x-a)^{\alpha}x^{\beta}(1-x)^{\gamma}$ in the interval $[0,1],$ with $a<0$ and $\alpha,\beta,\gamma>-1.$ In other words,
\begin{equation}
\int_{a}^{0}x^kP_{n,m}^{(\alpha,\beta,\gamma)}(x;a)(x-a)^{\alpha}|x|^{\beta}(1-x)^{\gamma}dx=0\quad k=0,\ldots,n-1,
\end{equation}
and
\begin{equation}
\int_{0}^{1}x^kP_{n,m}^{(\alpha,\beta,\gamma)}(x;a)(x-a)^{\alpha}x^{\beta}(1-x)^{\gamma}dx=0\quad k=0,\ldots,m-1.
\end{equation}

The Jacobi-Laguerre polynomials,  denoted by $L_{n,m}^{(\alpha,\beta)}(x;a),$ are multiple orthogonal polynomials, with respect to
$\omega_1(x)=(x-a)^{\alpha}|x|^{\beta}e^{-x}$ in $[a,0]$ and $\omega_2(x)=(x-a)^{\alpha}x^{\beta}e^{-x}$ in $[0,+\infty)$,
where $a<0$ and $\alpha,\beta>-1.$ This means that
\begin{equation}
\int_{a}^{0}x^kL_{n,m}^{(\alpha,\beta)}(x;a)(x-a)^{\alpha}|x|^{\beta}e^{-x}dx=0\quad k=0,\ldots,n-1
\end{equation}
and
\begin{equation}
\int_{0}^{+\infty}x^kL_{n,m}^{(\alpha,\beta)}(x;a)(x-a)^{\alpha}x^{\beta}e^{-x}dx=0\quad k=0,\ldots,m-1.
\end{equation}

The third family of multiple orthogonal polynomials are the  Laguerre-Hermite ones, denoted by $H_{n,m}^{(\beta)}(x),$ and orthogonal with respect to
the weight functions  $\omega_1(x)=|x|^{\beta}e^{-x^2}$ in the interval $(-\infty,0]$ and $\omega_2(x)=x^{\beta}e^{-x^2}$ in the interval $[0,+\infty),$ with $\beta>-1.$
This is equivalent to
\begin{equation}
\int_{a}^{0}x^kH_{n,m}^{(\beta)}(x)|x|^{\beta}e^{-x^2}dx=0\quad k=0,\ldots,n-1,
\end{equation}
and
\begin{equation}
\int_{0}^{1}x^kH_{n,m}^{(\beta)}(x)x^{\beta}e^{-x^2}dx=0\quad k=0,\ldots,m-1.
\end{equation}
These multiple orthogonal polynomials are connected by the following asymptotic relations:
\begin{equation}\label{laguerre-angelesco}
    L_{n,m}^{(\alpha,\beta)}(x,a)=\lim_{\gamma\rightarrow\infty}\gamma^{n+m}P_{n,m}^{(\alpha,\beta,\gamma)}(x/\gamma;a/\gamma)
\end{equation}
and
\begin{equation}\label{hermite-angelesco}
    H_{n,m}^{(\beta)}(x)=\lim_{\alpha\rightarrow\infty}\sqrt{\alpha}^{n+m}P_{n,m}^{(\alpha,\beta,\alpha)}(x/\sqrt{\alpha};-1).
\end{equation}

\section{Preliminary results}
The Jacobi-Angelesco polynomials obey the Rodrigues type formula
\begin{equation}\label{formuladerodrigues}
\omega(x;a,\alpha,\beta,\gamma)P_{n,n}^{(\alpha,\beta,\gamma)}(x;a)=c_n(\alpha,\beta,\gamma)\dfrac{d^{n}}{dx^{n}}\left(\omega(x;a,\alpha+n,\beta+n,\gamma+n)\right),
\end{equation}
where $c_n(\alpha,\beta,\gamma)=(-1)^n\left[(\alpha+\beta+\gamma+2n+1)_n\right]^{-1}$  and $\omega(x;a,\alpha,\beta,\gamma)=(x-a)^{\alpha}x^{\beta}(1-x)^{\gamma}.$
Hence,
\begin{equation}
\label{2.2}
\begin{split}
\omega(x;a,\alpha,\beta,\gamma)P_{n,n}^{(\alpha,\beta,\gamma)}(x;a)=\dfrac{c_n(\alpha,\beta,\gamma)}{c_{n-1}(\alpha+1,\beta+1,\gamma+1)}\hspace{2cm}\\
\times\dfrac{d}{dx}\left[c_{n-1}(\alpha+1,\beta+1,\gamma+1)\dfrac{d^{n-1}}{dx^{n-1}}\left(\omega(x;a,\alpha+n,\beta+n,\gamma+n)\right)\right]\\
=c_n(\alpha,\beta,\gamma)\dfrac{d}{dx}\left[\dfrac{\omega(x;a,\alpha+1,\beta+1,\gamma+1)P_{n-1,n-1}^{(\alpha+1,\beta+1,\gamma+1)}(x;a)}{c_{n-1}(\alpha+1,\beta+1,\gamma+1)}\right].
\end{split}
\end{equation}
Let us define the auxiliary functions
\begin{equation}
\begin{split}
R_{n-k,n-k}^{(\alpha+k,\beta+k,\gamma+k)}(x;a):=\dfrac{\omega(x;a,\alpha+k,\beta+k,\gamma+k)P_{n-k,n-k}^{(\alpha+k,\beta+k,\gamma+k)}(x;a)}{c_{n-k}(\alpha+k,\beta+k,\gamma+k)},
\end{split}
\end{equation}
for $k=0,\ldots n-1.$

Denote the zeros of the diagonal Jacobi-Angelesco  polynomial  $P_{n,n}^{(\alpha,\beta,\gamma)}(x;a)$ by $x_{n,j}(a,\alpha,\beta,\gamma)$, ordered in an increasing way according with
the index  $j$, where $j=1,\ldots,2n.$ Using the Rodrigues formula \eqref{formuladerodrigues} we obtain
\begin{equation}
\begin{split}
\omega(x;a,\alpha,\beta,\gamma)P^{(\alpha,\beta,\gamma)}_{n,n}(x;a)&=\dfrac{c_{n}(\alpha,\beta,\gamma)}{c_{n-1}(\alpha+1,\beta+1,\gamma+1)}\\
&\times\left[\omega^{\prime}(x;a,\alpha+1,\beta+1,\gamma+1)P^{(\alpha+1,\beta+1,\gamma+1)}_{n-1,n-1}(x)\right.\\
&\ \left.+\omega(x;a,\alpha+1,\beta+1,\gamma+1)P^{(\alpha+1,\beta+1,\gamma+1)\prime}_{n-1,n-1}(x)\right],
\end{split}
\end{equation}
so that
\begin{equation}
\begin{split}
\dfrac{c_{n-1}(\alpha+1,\beta+1,\gamma+1)P^{(\alpha,\beta,\gamma)}_{n,n}(x;a)}{c_{n}(\alpha,\beta,\gamma)\omega(x;a,\alpha,\beta,\gamma)P^{(\alpha+1,\beta+1,\gamma+1)}_{n-1,n-1}(x;a)}&=\dfrac{\omega^{\prime}(x;a,\alpha+1,\beta+1,\gamma+1)}{\omega(x;a,\alpha+1,\beta+1,\gamma+1)}\\
&\ +\dfrac{P^{(\alpha+1,\beta+1,\gamma+1)\prime}_{n-1,n-1}(x;a)}{P^{(\alpha+1,\beta+1,\gamma+1)}_{n-1,n-1}(x;a)}.
\end{split}
\end{equation}
Therefore, $P^{(\alpha,\beta,\gamma)}_{n,n}(x;a)=0$ for some $x\in(a,0)\cup(0,1)$ if and only if
$$
\dfrac{P^{(\alpha+1,\beta+1,\gamma+1)\prime}_{n-1,n-1}(x;a)}{P^{(\alpha+1,\beta+1,\gamma+1)}_{n-1,n-1}(x;a)}+\dfrac{\omega^{\prime}(x;a,\alpha+1,\beta+1,\gamma+1)}{\omega(x;a,\alpha+1,\beta+1,\gamma+1)}=0,
$$
or equivalently,
\begin{equation}\label{eq_fundamental}
\begin{split}
\sum_{j=1}^{2(n-1)}\dfrac{1}{x-x_{n-1,j}(a,\alpha+1,\beta+1,\gamma+1)}
+\dfrac{\alpha+1}{x-a}+\dfrac{\beta+1}{x}-\dfrac{\gamma+1}{1-x}=0.
\end{split}
\end{equation}
Let us denote the left-hand side of \eqref{eq_fundamental} by $f_{n-1}(x;a,\alpha+1,\beta+1,\gamma+1).$
Observe that another way to write the function $ f_{n-1}(x;a,\alpha+1,\beta+1,\gamma+1)$ is
\begin{equation}
    f_{n-1}(x;a,\alpha+1,\beta+1,\gamma+1)=\dfrac{R_{n,n}^{(\alpha,\beta,\gamma)}(x;a)}{R_{n-1,n-1}^{(\alpha+1,\beta+1,\gamma+1)}(x;a)}.
\end{equation}
The zeros of $P^{(\alpha,\beta,\gamma)}_{n,n}(x)$ coincide with the zeros of $f_{n-1}(x;a,\alpha+1,\beta+1,\gamma+1)$ belonging to the set  $(a,0)\cup(0,1).$
For the sake of brevity, the intervals bounded by the zeros of $P^{(\alpha,\beta,\gamma)}_{n,n}(x)$ and the points $a,\ 0$ and $1,$ are denoted by $I_{n,j}(a,\alpha,\beta,\gamma),$ with $j=1,\ldots,2n+1.$
We prove a lemma concerning to the behaviour of the functions $f_{n-1}(x;a,\alpha+1,\beta+1,\gamma+1)$ in these intervals.
\begin{lemma}\label{Lema2}
For every $n\in \mathbb{N}$ the function $f_{n-1}(x;a,\alpha+1,\beta+1,\gamma+1)$ obeys the following properties:
\begin{itemize}
\item[$i)$] $f_{n-1}(x;a,\alpha+1,\beta+1,\gamma+1)$ is decreasing in the intervals $I_{n-1,j}(a,\alpha+1,\beta+1,\gamma+1)$, for $j=0,\ldots,2n-1;$
\item[$ii)$] The limit relations
$$\lim_{x\rightarrow x^{\pm}_{n-1,j}(a,\alpha+1,\beta+1,\gamma+1)}f_{n-1}(x;a,\alpha+1,\beta+1,\gamma+1)=\pm\infty,$$
hold for  $j=1,\ldots,2n-2.$ Moreover, the same limits hold at the points $a,\ 0$ and $1;$
\item[$iii)$] $f_{n-1}(x;a,\alpha+1,\beta+1,\gamma+1)$ possesses a unique zero in each interval $I_{n-1,j}(a,\alpha+1,\beta+1,\gamma+1)$, for $j=0,\ldots,2n-1$;
\item[$iv)$] Both the positive and the negative  zeros of  $P_{n-1,n-1}^{(\alpha+1,\beta+1,\gamma+1)}(x)$ and $P_{n,n}^{(\alpha,\beta,\gamma)}(x)$ interlace in the sense that
$$
x_{n,j}(\alpha,\beta,\gamma)<x_{n-1,j}(\alpha+1,\beta+1,\gamma+1)<x_{n,j+1}(\alpha,\beta,\gamma),\ \  j=1,\ldots,n-1
$$
and 
$$
x_{n,j}(\alpha,\beta,\gamma)<x_{n-1,j-1}(\alpha+1,\beta+1,\gamma+1)<x_{n,j+1}(\alpha,\beta,\gamma),\ \ j=n+1,\ldots,2n-1.
$$
\end{itemize}
\end{lemma}
\begin{proof}
For the proof of $i)$ observe that the derivative
\begin{equation*}
\begin{split}
\dfrac{\partial}{\partial x}\left[f_{n-1}(x;a,\alpha+1,\beta+1,\gamma+1)\right]&=-\sum_{j=1}^{2n}\dfrac{1}{(x-x_{n-1,j}(a,\alpha+1,\beta+1,\gamma+1))^{2}}\\
&\hspace{0.5cm} -\dfrac{\alpha+1}{(x-a)^2}-\dfrac{\beta+1}{x^2}-\dfrac{\gamma+1}{(1-x)^2},
\end{split}
\end{equation*}
is negative in the intervals $I_{n-1,j}(a,\alpha+1,\beta+1,\gamma+1),$ for every $j=1,\ldots 2n-1.$ Hence, $f_{n-1}(x,a,\alpha+1,\beta+1,\gamma+1)$ are decreasing in the intervals $I_{n-1,j}(a,\alpha+1,\beta+1,\gamma+1)$ for every $j=1,\ldots,2n+1$ and $n\geq 0$.

Item $ii)$ follows from
\begin{eqnarray*}
\lim_{x\rightarrow x^{\pm}_{n-1,j}(a,\alpha+1,\beta+1,\gamma+1)}f_{n-1}(x;a,\alpha+1,\beta+1,\gamma+1)\\
=\lim_{x\rightarrow x^{\pm}_{n-1,j}(a,\alpha+1,\beta+1,\gamma+1)}\dfrac{1}{x-x_{n-1,j}(a,\alpha+1,\beta+1,\gamma+1)},
\end{eqnarray*}
which readily implies
$$\lim_{x\rightarrow x^{\pm}_{n-1,j}(a,\alpha+1,\beta+1,\gamma+1)}f_{n-1}(x;a,\alpha+1,\beta+1,\gamma+1)=\pm\infty,$$
for $j=1,\ldots,2n-2$. Similarly, since $\alpha+1,\ \beta+1$ and $\gamma+1$ are positive, we obtain the same limit relations at the points $a,\ 0$ and $1$.

According to $i)$ and $ii)$ the function $f_{n-1}(x;a,\alpha+1,\beta+1,\gamma+1)$ takes all real values exactly once in each interval $I_{n-1,j}(a,\alpha+1,\beta+1,\gamma+1)$, $j=1,\ldots,2n-1$. The monotonicity yields  that there is a unique
$\overline{x}\in I_{n-1,j}(a,\alpha+1,\beta+1,\gamma+1)$ such that $$f_{n-1}(\overline{x};a,\alpha+1,\beta+1,\gamma+1)=0$$
and this proves $iii)$.

Finally we remark that the statement of item $iii)$, together with \eqref{2.2}, implies that both the positive and negative  zeros of  $P_{n-1,n-1}^{(\alpha+1,\beta+1,\gamma+1)}(x)$ and $P_{n,n}^{(\alpha,\beta,\gamma)}(x)$ interlace and $iv)$ holds.
\end{proof}

\section{Monotonic behaviour of the zeros of Jacobi-Angelesco, Jacobi-Laguerre and Laguerre-Hermite polynomials}

The problem about the monotonic behaviour of zeros of multiple orthogonal polynomials was formulated by Ismail \cite[Problem 24.1.5]{Mourad}.
The lack of a result like Markov's theorem (see \cite[Theorem 6.12.1]{Szego})  or other tools, such as Sturm's comparison theorem or the
Hellmann--Feynman theorem, that are commonly employed to study monotonicity of zeros of orthogonal polynomials makes Ismail's problem hard.

In this note we prove a specific result concerning to the monotonic behaviour of the zeros of Jacobi-Angelesco polynomials in the so-called diagonal case.
It is proved that all zeros are monotonically increasing functions of $\alpha$ and decreasing functions of $\gamma$. The monotonicity
with respect to $\beta$ is more peculiar: we believe that the $n$ zeros in $[a,0]$ decrease while the $n$ zeros in $[0,1]$ increase when $\beta$
increases. We are able to prove this fact in the particular case when $a=-1$ and $\alpha=\gamma$. The monotonic behaviour of Jacobi-Laguerre and Laguerre-Hermite polynomials will be obtained as a corollary.

\begin{theorem}\label{Main_theorem}
Let $n\in \mathbb{N}$ and $P^{(\alpha,\beta,\gamma)}_{n,n}(x),\ n\geq1$ be the corresponding diagonal  Jacobi-Angelesco polynomials. Then:
\begin{itemize}
\item[$a)$] all zeros of $P^{(\alpha,\beta,\gamma)}_{n,n}(x)$ are increasing with respect to the parameter $\alpha;$
 \item[$b)$] the negative zeros of  $P^{(\alpha,\beta,\gamma)}_{n,n}(x)$  are decreasing with respect to the parameter  $\beta$ and the positive zeros are increasing with respect to $\beta,$ provided $a=-1$ and $\alpha=\gamma;$
\item[$c)$] the zeros of $P^{(\alpha,\beta,\gamma)}_{n,n}(x)$ are decreasing with respect to the parameter $\gamma.$
\end{itemize}
\end{theorem}

\begin{proof}
When it is clear from the context, we skip the parameters in some of the expressions, especially if they are considered to be fixed. For instance, in $a)$ we are interested in the behaviour of the zeros with respect to the parameter $\alpha$, so that $P_{n,n}^{(\alpha,\beta,\gamma)}(x;a)$ will be abbreviated to $P_{n,n}^{(\alpha)}(x).$

We prove only item $a)$ and in item $b)$ we shall deal only with the monotonic behaviour of the negative zeros. The arguments in the other cases are quite similar.

First we establish $a)$ by induction with respect to $n$. For this purpose we define the functions
\begin{equation}\label{definicaoh}
 h_{k}(x;\alpha+n-k,\delta):=f_{k}(x;\alpha+n-k)-f_{k}(x;\alpha+\delta+n-k)
\end{equation}
for $0\leq k\leq n-1,\ n\geq 1$ and $\delta>0.$

For $n=1$ the zeros of $P_{1,1}^{(\alpha)}(x)$ coincide with the zeros of $f_0(x;\alpha+1).$
Since
$$h_{0}(x;\alpha+1,\delta)=-\dfrac{\delta}{x-a},$$
we obtain $h_{0}(x;\alpha+1,\delta)<0,$ for every $x>a,$ and then
$$f_{0}(x_{1,j}(\alpha);\alpha+\delta+1)>0,\quad j=1\text{ or }j=2.$$
Using item $iii)$ of Lemma \ref{Lema2} we conclude that there are no zeros of $f_0(x;\alpha+\delta+1)$  and consequently of $P^{(\alpha+\delta)}_{1,1}(x)$ in the intervals $\left(a,x_{1,1}(\alpha)\right]$ or $\left(0,x_{1,2}(\alpha)\right],$ because $f_0(x;\alpha+\delta+1)$ doesn't change him sign in these intervals. Thus
$$
x_{n,j}(\alpha)<x_{n,j}(\alpha+\delta),\ \ \mathrm{for}\ \  j=1,2.
$$
Therefore, the zeros of  $P^{(\alpha)}_{1,1}(x)$ are increasing functions of the parameter $\alpha.$

Suppose that the zeros of  $P_{n-1,n-1}^{(\alpha+1)}(x)$ are increasing with respect to $\alpha,$ for  $n \geq 2$  arbitrarily fixed, namely that
$$x_{n-1,j}(\alpha+\delta+1)=x_{n-1,j}(\alpha+1)+\xi_{n-1,j}(\delta),$$
for $j=1,\ldots,2n-2$ with $\xi_{n-1,j}(\delta)>0,$ for every $\delta>0.$

We shall analyze how the fact that $\alpha$ increase affects the zeros of $P^{(\alpha}_{n,n}(x),$ having in mind that under the induction hypothesis that the zeros of $P_{n-1,n-1}^{(\alpha+1)}$ increase with $\alpha.$
In order to this we verify the sign of the function  $h_{n-1}(x;\alpha+1,\delta)$ in the intervals
\begin{equation}
\begin{matrix}
(a,x_{n-1,j}(\alpha+1));\\
(x_{n-1,j}(\alpha+\delta+1),x_{n-1,j+1}(\alpha+1)),\ j=1,\ldots,n-2;\\
(x_{n-1,n-1}(\alpha+\delta+1),0);\\
(0,x_{n-1,n}(\alpha+1));\\
(x_{n-1,j}(\alpha+\delta+1),x_{n-1,j+1}(\alpha+1)),\ j=n,\ldots,2n-3;\\
(x_{n-1,2n-2}(\alpha+\delta+1),1).
\end{matrix}
\end{equation}

Since the positive and the negative zeros of $P_{n-1,n-1}^{(\alpha+1)}(x)$ and $P_{n,n}^{(\alpha)}(x)$  are interlacing, the continuity of the zeros with respect to the parameter   $\alpha$ allows us to choose  $\delta>0$ small enough such that
\begin{equation}
x_{n-1,j}(\alpha+1)<x_{n-1,j}(\alpha+\delta+1)<x_{n,j+1}(\alpha+1)
\end{equation}
for $j=1,\ldots,n-1$ and
\begin{equation}
x_{n-1,j}(\alpha+1)<x_{n-1,j}(\alpha+\delta+1)<x_{n,j+2}(\alpha+1)
\end{equation}
for $j=n,\ldots,2n-2.$
Since
$$x_{n,i}(\alpha)-x_{n-1,j}(\alpha+1)>x_{n,i}(\alpha)-x
_{n-1,j}(\alpha+1)-\xi_{n-1,j}(\delta
) $$
for $j=1,\ldots,2n-2$ and $i=1,\ldots,2n,$ we obtain
\begin{equation}
\begin{split}
\sum_{j=1}^{2n-2}\dfrac{1}{x_{n,i}(\alpha)-x_{n-1,j}(\alpha+1)}<\sum_{j=1}^{2n}\dfrac{1}{x_{n,i}(\alpha)-x
_{n-1,j}(\alpha+\delta+1)},
\end{split}
\end{equation}
so that
\begin{equation}
    h_{n-1}(x_{n,i}(\alpha);\alpha+1,\delta)<0,
\end{equation}
for $i=1,\ldots,2n.$
%see Figure \ref{fig:h_n_alpha}.
Hence,
$$f_{n-1}(x_{n,i}(\alpha);\alpha+\delta+1)>f_{n-1}(x_{n,i}(\alpha);\alpha+1),
$$
for $i=1,\ldots,2n.$ Since $f_{n-1}(x_{n,i}(\alpha);\alpha+1)=0$ for $i=1,\ldots,2n,$ we conclude that
$$f_{n}(x_{n,i}(\alpha);\alpha+\delta+1)>0,
$$
for $i=1,\ldots,2n.$

Apply again $iii)$ of the Lemma \ref{Lema2} to conclude that there are no zeros of $P^{(\alpha+\delta)}_{n,n}(x)$ in the intervals
$$\left(a,x_{n,1}(\alpha)\right],$$
$$\left(x_{n-1,j}(\alpha+1),x_{n,j+1}(\alpha)\right],\ \ j=1,\ldots,n-1,$$
$$\left(0,x_{n,n+1}(\alpha)\right],$$
and
$$\left(x_{n-1,j}(\alpha+1),x_{n,j+2}(\alpha)\right],\ j=n,\ldots,2n-2$$
because $f_{n-1}(x;\alpha+\delta+1)$ does not change sign in these intervals.

Therefore,  the zeros of  $P^{(\alpha+\delta)}_{n,n}(x)$ must belong to the intervals
$$x_{n,j+1}(\alpha+\delta)\in\left(x_{n,j+1}(\alpha),x_{n-1,j+1}(\alpha+1)\right),\ j=0,\ldots,n-2,$$
$$x_{n,n}(\alpha+\delta)\in\left(x_{n,n}(\alpha),0\right),$$
$$x_{n,j+2}(\alpha+\delta)\in\left(x_{n,j+2}(\alpha),x_{n-1,j+1}(\alpha+1)\right),\ j=n-1,\ldots,2n-3$$
and
$$x_{n,2n}(\alpha+\delta)\in\left(x_{n,2n}(\alpha),1\right).$$
In other words, the zeros of $P^{(\alpha+\delta)}_{n,n}(x)$ are located ``to the right'' with respect to the corresponding zeros of $P^{(\alpha)}_{n,n}(x)$.

We prove $b)$ only for the zeros that belong to the interval $(-1,0)$ because the proof concerning those in $(0,1)$ is similar for reasons of symmetry.
Indeed, when $\alpha=\gamma=\lambda-1/2$ the Rodrigues formula \eqref{formuladerodrigues} yields
\begin{equation*}
    P_{n,n}^{(\lambda-1/2,\beta,\lambda-1/2)}(x)=\sum_{i=0}^{n}\binom{n}{i}(-\beta-n)_{i}\kappa_{n-i}(\lambda)C^{(\lambda+i)}_{n-i}(x)x^{n-i},
\end{equation*}
where $C^{(\lambda)}_n(x)$  are the classical Gegenbauer polynomials and $$\kappa_n(\lambda)=\dfrac{(-1)^n\left(2\lambda\right)_n}{2^{n}\left(\lambda+1/2\right)_n n!}.$$
Since $x^{n-i}C^{(\lambda+i)}_{n-i}(x)$ are even polynomials, we conclude that $P_{n,n}^{(\lambda-1/2,\beta,\lambda-1/2)}(x)$ are even polynomials too. Hence, in this case the zeros of
$P_{n,n}^{(\lambda-1/2,\beta,\lambda-1/2)}(x)$ are symmetric with respect to the origin.

The proof of  $b)$  also goes by induction with respect to $n$. Let us define
\begin{equation}\label{definicaohdelta}
 h_{k}(x;\beta+n-k,\delta):=f_{k}(x;\beta+n-k)-f_{k}(x;\beta+\delta+n-k)
\end{equation}
for $0\leq k\leq n-1,\ n\geq 1$ and $\delta>0.$

For $n=1$  the negative zeros of $P_{1,1}^{(\beta)}(x)$ are the zeros of $f_0(x;\beta+1)$. Since
$$h_{0}(x;\beta+1,\delta)=-\dfrac{\delta}{x}$$
then $h_{0}(x;\beta+1,\delta)>0,$ for every $-1<x<0,$ so that
$$f_{0}(x_{1,1}(\beta);\beta+\delta+1)<0.$$
Using $iii)$ of Lemma \ref{Lema2}, we conclude that $f_0(x;\beta+\delta+1)$, and consequently $P^{(\beta+\delta)}_{1,1}(x)$, does not vanish in the interval $\left[x_{1,1}(\alpha),0\right],$ because these functions do not change sign there.
Therefore,
$$x_{1,1}(\beta)>x_{1,1}(\beta+\delta).$$
This is equivalent to the fact that the negative zero of  $P^{(\beta)}_{1,1}(x)$ decreases when the parameter $\beta$ increases.

Suppose that the negative zeros of $P_{n-1,n-1}^{(\beta+1)}(x)$ are decreasing with respect to the parameter $\beta,$ for some $n \geq 2$  arbitrarily  fixed, namely that
$$x_{n-1,j}(\beta+\delta+1)=x_{n-1,j}(\beta+1)-\xi_{n-1,j}(\delta),$$
for $j=1,\ldots,n-1$ with $\xi_{n-1,j}(\delta)>0,$ for $\delta>0.$
As a consequence of the symmetry of the zeros, we have $x_{n-1,j}(\beta+\delta+1)=-x_{n-1,2n-2-j}(\beta+\delta+1),$
so if we take $\delta>0$ small enough, such that
\begin{equation}
x_{n,j}(\beta+1)<x_{n-1,j}(\beta+\delta+1)<x_{n-1,j}(\beta+1),
\end{equation}
for $j=1,\ldots,n-1,$ we obtain
\begin{equation}\label{hbeta}
\begin{split}
    h_{n-1}(x;\beta+1,\delta)=\sum_{j=1}^{n-1}2x\left[\dfrac{1}{x^2-x_{n-1,j}^2(\beta+1)}-\dfrac{1}{x^2-x_{n-1,j}^2(\beta+\delta+1)}\right]
    -\dfrac{\delta}{x}.
\end{split}
\end{equation}
Every term of the latter sum vanishes at the origin and is positive for $x\in(-1,x_{n-1,j}(\beta+\delta+1))\cup(x_{n-1,j}(\beta+1),0).$
This implies that
\begin{equation}
    h_{n-1}(x_{n,i}(\beta);\beta+1,\delta)>0,\ \ i=1,\ldots,n.
\end{equation}
Hence,
$$
f_{n-1}(x_{n,i}(\beta);\beta+\delta+1)<f_{n-1}(x_{n,i}(\beta);\beta+1),\ i=1,\ldots,n.
$$
The fact that $f_{n-1}(x_{n,i}(\beta);\beta+1)=0,\ i=1,\ldots,n,$ implies
$$
f_{n-1}(x_{n,i}(\beta);\beta+\delta+1)<0,\ i=1,\ldots,n.
$$

Again, like in the first step, using $iii)$ of Lemma \ref{Lema2}, we conclude that $P^{(\beta+\delta)}_{n,n}(x)$ does not posses zeros in the intervals
$$\left(x_{n,j}(\beta),x_{n-1,j}(\beta+1)\right],\ j=1,\ldots,n-1,$$
and
$$\left(x_{n,n}(\beta),0\right]$$
because there are no sign changes of there.

Therefore, the negative zeros of $P^{(\beta+\delta)}_{n,n}(x)$ belong to the intervals
$$x_{n,1}(\beta+\delta)\in\left(-1,x_{n,1}(\beta)\right),$$
and
$$x_{n,j+1}(\beta+\delta)\in\left(x_{n-1,j}(\beta+1),x_{n,j+1}(\beta)\right),\ j=1,\ldots,n-1.$$
In other words, the zeros of $P^{(\beta+\delta)}_{n,n}(x)$ are all located ``to the left''  with respect to the zeros of $P^{(\beta)}_{n,n}(x).$  This completes the prove of $b).$
\end{proof}

It is worth mentioning that an alternative way to prove the monotonicity of the zeros of $P^{(\alpha,\beta,\gamma)}_{n,n}(x)$ with respect to $\alpha$ and $\gamma$ is to employ 
a modification of a classical result of Markov \cite{Mar} (see also \cite[Lemma 2.7.1]{Riv} and \cite[Lemma 1]{DKD}) which states that all the zeros of the derivative of a 
polynomial with real zeros are monotonic functions of each zero of the polynomial itself.

As direct consequences of Theorem \ref{Main_theorem}, \eqref{laguerre-angelesco} and \eqref{hermite-angelesco} we obtain:
\begin{cor}
The zeros of Jacobi-Laguerre polynomials $L_{n,n}^{(\alpha,\beta)}(x,a)$ are increasing functions of the parameter $\alpha.$
\end{cor}
and
\begin{cor}
The negative zeros of Laguerre-Hermite polynomials $H_{n,n}^{(\beta)}(x)$ are decreasing functions and its positive zeros are increasing functions of the parameter $\beta.$
\end{cor}

\bibliographystyle{amsplain}

\begin{thebibliography}{10}

\bibitem{DKD} 
D. K. Dimitrov, On a conjecture concerning monotonicity of zeros
of ultraspherical polynomials, J. Approx. Theory 85 (1996), 88--96.

\bibitem{Mourad}
M. E. H.  Ismail, Classical and Quantum Orthogonal Polynomials in One Variable, Cambridge University Press, first paperback edition, 2009.

\bibitem{Mar} V. Markov, On functions least deviated from zero on a given interval,
St. Petersburg, 1892 [in Russian];  \"Uber Polynome die in einen gegebenen Intervalle
m\"oglichst wenig von Null abweichen, (W. Markoff, transl.), Math. Ann. 77 (1916),
213--258. [German transl.]

\bibitem{Riv}  T. J. Rivlin, Chebyshev Polynomials: From Approximation Theory to Algebra and
Number Theory, 2nd ed., Wiley, New York, 1990.


\bibitem{Szego}
G.  Szego, \textit{Orthogonal Polynomials}, AMS-Colloquium publications, volume XXIII, forth edition, 1975.

\bibitem{Walter}
W. Van Assche and E. Coussement, Some classical multiple orthogonal polynomials, J. Comput. Appl. Math. 127 (2001), 317--347.

\end{thebibliography}

\end{document}